\documentclass[12pt, onesided]{amsart}

\usepackage{amsmath, amsthm, amssymb, graphicx, amsfonts, hyperref}
\usepackage{color}
\usepackage{xypic, pinlabel}
\usepackage[all]{xy}

\newtheorem{lemma}{Lemma}[section]
\newtheorem{thm}[lemma]{Theorem}
\newtheorem{prop}[lemma]{Proposition}
\newtheorem{cor}[lemma]{Corollary}

\usepackage[margin=3cm, marginpar=2cm]{geometry}

\theoremstyle{definition}

\newtheorem{rmk}[lemma]{Remark}
\newtheorem{quest}[lemma]{Question}

\newtheorem{example}[lemma]{Example}

\title[Rational homology balls and Casson--Gordon invariants]{Handle decompositions of rational homology balls and Casson--Gordon invariants}

\author{Paolo Aceto}
\address{R\'enyi Institute of Mathematics, Budapest, Hungary}
\email{aceto.paolo@renyi.mta.hu}
\author{Marco Golla}
\address{Department of Mathematics, Uppsala University, Sweden}
\email{marco.golla@math.uu.se}
\author{Ana G. Lecuona}
\address{Aix Marseille Univ., CNRS, Centrale Marseille, I2M, Marseille, France}
\email{ana.lecuona@univ-amu.fr}

\date{}

\newcommand{\Z}{\mathbb{Z}}
\newcommand{\Zm}{C_{m}}
\newcommand{\Zk}{C_{k}}
\newcommand{\Q}{\mathbb{Q}}
\newcommand{\C}{\mathbb{C}}

\newcommand{\de}{\partial}
\renewcommand{\tilde}{\widetilde}
\renewcommand{\phi}{\varphi}

\DeclareMathOperator{\rk}{rk}
\DeclareMathOperator{\coker}{coker}

\begin{document}

\maketitle
\begin{abstract}
Given a rational homology sphere which bounds rational homology balls, we investigate the complexity of these balls as measured by the number of 1-handles in a handle decomposition.
We use Casson--Gordon invariants to obtain lower bounds which also lead to lower bounds on the fusion number of ribbon knots.
We use Levine--Tristram signatures to compute these bounds and produce explicit examples.
\end{abstract}

\section{Introduction}
Given two concordant knots it is natural to ask how complicated a concordance between them must be.
A similar question can be asked about (rational) homology cobordant 3-manifolds and cobordisms between them.
Very little is known about these two simple and natural questions.

In the context of knot concordance a natural notion of complexity already considered by several authors \cite{fusion} is that of the \emph{fusion number} of a ribbon knot, i.e.\ the minimal number of 1-handles needed to construct a ribbon disc. The analogous notion for homology spheres bounding a homology ball, in the integral and rational case, is the minimum number of 1-handles needed to construct such a ball.
These numerical invariants encode deep 4-dimensional information on knots and 3-manifolds and are extremely hard to compute.
Motivation in this direction comes also from analogous questions which are purely 4-dimensional.
One of the oldest open problems in smooth 4-manifold topology asks if it is true that every smooth simply-connected 4-manifold admits a handle decomposition with no 1-handles.

In this paper we investigate the complexity of rational homology balls (as measured by the number of handles in their handle decompositions) bounded by a given rational homology sphere. More precisely we consider the following question.

\begin{quest}
Let $Y$ be a rational homology sphere which bounds a rational homology ball.
What is the minimal number of 1-handles needed to realise a rational homology ball bounded by $Y$?
What if we restrict to those rational homology balls constructed only with handles of index at most 2?
\end{quest}

We provide lower bounds on these numbers using Casson--Gordon signature invariants, which associate to a rational homology sphere $Y$ and a character $\varphi:H_1(Y;\Z)\rightarrow\mathbb{C}^*$ the rational number $\sigma(Y,\varphi)$.
One of the key features of our approach is the use of non prime order characters.
We relate this number to handle decompositions of rational homology balls bounding $Y$ via the following statement.

\begin{thm}\label{t:mainintro}
Let $Y$ be a $3$-manifold that bounds a rational homology ball $W$, and let $\varphi: H_1(Y)\to\mathbb{C}^*$ be a nontrivial character that factors through $H_1(W)$.
Every handle decomposition of $W$ contains at least $|\sigma(Y,\phi)|-1$ odd-index handles.
\end{thm} 

A particularly simple case in which we can use the above theorem is when any given character on the 3-manifold factors through  \emph{any} rational homology ball bounded by it. One example of this situation is described in the following corollary, where $C_{k}$ denotes the cyclic group with $k$ elements, seen as the subgroup of $\C^{*}$ generated by
a root of unity of order $k$. Notice that in the next corollary the only restriction we are imposing on $H_{1}(Y)$ is that this group is cyclic. If $Y$ bounds a rational homology ball, it follows from the long exact sequence of the pair that the order of $H_{1}(Y)$ is a square.  

\begin{cor}
Let $Y$ be a $3$-manifold with $H_1(Y)$ cyclic of order $m^2$ 
and $\phi:H_1(Y) \to C_{k}$ be a nontrivial character with $k | m$.
If $W$ is a rational homology ball with $\partial W = Y$, then every handle decomposition of $W$ contains at least $|\sigma(Y,\phi)|-1$ odd-index handles.
\end{cor}

In order to produce specific examples we need an efficient way to compute Casson--Gordon invariants. Using work of Cimasoni and Florens and focusing on 3-manifolds obtained via Dehn surgery on knots we reduce our problem to a computation of Levine--Tristram signatures. We denote by $S^{3}_{r}(K)$ the manifold obtained by  performing a surgery of slope $r$ on a knot $K\subset S^{3}$ and by $\sigma_{K}(\omega)$ the Levine--Tristram signature of $K$ evaluated at $\omega$.

\begin{prop}
If $S^3_{m^2}(K)$ bounds a rational homology ball $W$ with one $1$-handle and no $3$-handles, then $|1-\sigma_K(e^{2a\pi i/m})-2a(m-a)| \le 1$ for every $1\le a < m$ such that $(a,m)=1$.
\end{prop}

Note that $S^3_{m^2}(K)$ bounds a rational homology ball with one 1-handle and no 3-handles if and only if it can be obtained via Dehn surgery on a knot in $S^1\times S^2$ and therefore we obtain an obstruction for this last property as well.

The examples obtained include the following:
\begin{itemize}
\item the connected sum of lens spaces $L(25,21)\# L(4,3)$ bounds no rational homology ball with a single 1-handle and it bounds one built with two 1-handles and two 2-handles;
\item the 3-manifold $S^3_{400}(T_{4,25;2,201})$ (here $T_{4,25;2,201}$ is the $(2,201)$-cable of the torus knot $T_{4,25}$) bounds no rational homology ball with a single 1-handle and it bounds one built with three 1-handles and three 2-handles.
\end{itemize}

Finally we use this machinery to provide lower bounds on the \emph{fusion number} of ribbon knots, 
i.e.\ the minimal number of 1-handles used to construct a ribbon disc (see Corollary~\ref{c:ribbon}).
This is done by looking at the double cover of the 4-ball, branched over any ribbon disc; this is well known to be a rational homology ball built only with handles of index at most 2. Moreover, the number of 1-handles used is the number of bands in the ribbon disc, and hence we can apply Theorem~\ref{t:mainintro} to give a bound on the number of bands.
In fact, other bounds can be given by looking at cyclic covers whose order is a prime power.

\subsection*{Organisation of the paper}
In Section \ref{main} we develop the lower bounds based on Casson--Gordon signatures and, in specific situations, we relate these invariants to Tristram--Levine signatures. In Section \ref{examples} we give some examples.

\subsection*{Acknowledgments} This work began when the first and third authors were visiting
Uppsala University; part of this work was carried out while the authors were at the Centro de Giorgi in Pisa; we would like to thank both institutions for their hospitality.
The first author is supported by the ERC Advanced Grant LDTBud.
The second author acknowledges support from the Alice and Knut Wallenberg Foundation and from the European Research Council (ERC) under the European Union's Horizon 2020 research and innovation programme (grant agreement No 674978).
The third author is partially supported by the Spanish GEOR MTM2011-22435.
We would also like to thank Maciej Borodzik, Marco Marengon, Brendan Owens, and Andr\'as Stipsicz for stimulating conversations, and the referee for their useful suggestions.


\section{Casson--Gordon signatures and handle decompositions}\label{main}

We briefly recall the definition of Casson--Gordon signature invariants~\cite{CassonGordon} and set up some notation.
In what follows $H_*(X)$ will denote the homology of $X$ with integer coefficients and $\Zm$ the cyclic group of $m$ elements.

Let $(Y,\phi)$ be a rational homology 3-sphere with a multiplicative character $\phi: H_1(Y)\to C_m \subset \C^*$.
Since the bordism group $\Omega_3(K(\Zm,1))$ is finite for each $m\in\Z_+$, there is $r\in\Z_{+}$ such that $r$ copies of $(Y,\phi)$ bound a pair $(X,\psi)$, where $X$ is a 4-manifold and $\psi: H_1(X)\to\Zm\subset\C^{*}$ restricts to $\phi$ on each of the $r$ boundary components.
Note that we make no assumption that $\phi$ is surjective onto $C_m\subset \C^*$.

Let $\tilde X$ denote the $m$-fold cover of $X$ corresponding to $\psi$ with group of deck transformations isomorphic to $\Zm$. This action induces a $\Z[\Zm]$-module structure on $H_{2}(\tilde X)$. Recall that given $\zeta_{m}$, a primitive root of unity of order $m$, the cyclotomic field $\Q(\zeta_{m})$ is a natural $\Z[\Zm]$-module and we can define the twisted homology group
\[
H_{2}^{\psi}(X;\Q(\zeta_{m})):=H_{2}(\tilde X;\Q)\otimes_{\Z[\Zm]}\Q(\zeta_{m}).
\]
This group admits a $\C$-valued Hermitian intersection form whose signature will be denoted by $\sigma^{\psi}(X)$. The signature of the standard intersection pairing on $H_{2}(X)$ will be denoted by $\sigma (X)$.
The \emph{Casson--Gordon signature invariant} of the pair $(Y,\phi)$ is given by the difference:
\begin{equation}\label{e:CGdef}
\sigma(Y,\phi) := \frac1r\left(\sigma^\psi(X) - \sigma(X)\right).
\end{equation}
Our main result provides a bound on the complexity of rational homology balls in terms of their handle decompositions. The proof of this result is very similar in nature to the one in the original paper of Casson and Gordon~\cite[Theorem 1]{CassonGordon} but with a different application in mind.

\begin{thm}\label{t:technical}
Let $Y$ be a $3$-manifold that bounds a rational homology ball $W$, and let $\phi: H_1(Y)\to\C^*$ be a nontrivial character that factors through $H_1(W)$.
Every handle decomposition of $W$ contains at least $|\sigma(Y,\phi)|-1$ odd-index handles.
\end{thm}

\begin{proof}
Let $\psi: H_1(W)\to \C^*$ be a character that extends $\phi$, namely, $\phi = \psi\circ i_*$ where $i: Y \hookrightarrow W$ is the inclusion. We shall use the manifold $W$ to compute $\sigma(Y,\phi)$ as in \eqref{e:CGdef}. In this case $r=1$ and since $W$ is a rational homology ball, $H_2(W)$ is torsion and hence $\sigma(W) = 0$. Therefore, in \eqref{e:CGdef} we are only concerned with the first summand, $\sigma^\psi(W)$.

We denote by $\tilde W$ the covering associated to $\psi$ and fix $m$ to be the order of $\psi(H_{1}(W))\subset\C^{*}$. Any cell decomposition of $W$ induces a chain complex of the covering $C_*(\tilde W)$, which we view as generated over $\Z[\Zm]$ by one lift of each cell in the given decomposition of $W$.
The module structure allows us to consider the twisted chain complex
\[
C^{\psi}_k(W;\Q(\zeta_{m})) := C_k(\tilde W)\otimes_{\Z[\Zm]}\Q(\zeta_{m}),
\]
with associated homology $H^\psi_*(W;\Q(\zeta_{m}))$ and Euler characteristic $\chi^{\psi}(W)$. Note that, since $W$ is a rational homology ball and has $\chi(W) = 1$, then also $\chi^\psi(W) = 1$. Moreover, observe that the $k$-th twisted Betti number $b_k^\psi(W) := \rk H_k^\psi(W;\Q(\zeta_{m}))$ of $W$ is bounded from above by the number $n_k$ of $k$-cells in the decomposition of $W$.

The quantity $\sigma^\psi(W)$ in formula \eqref{e:CGdef} is obviously bounded by $b_2^\psi(W)$ and since $\psi$ is nontrivial by assumption, $H_0^\psi(W;\Q(\zeta_{m})) = 0$. Therefore, since 
\[
1=\chi^\psi(W)=- b_1^\psi(W) + b_2^\psi(W) - b_3^\psi(W),
\]
we have
\[
|\sigma(Y,\phi)|=|\sigma^\psi(W)| \le b_2^\psi(W) = 1+b_1^\psi(W) + b_3^\psi(W) \le 1+n_1+n_3,
\]
as desired.
\end{proof}

The statement of Theorem~\ref{t:technical} requires that a character defined on a 3-manifold extend over a rational homology 4-ball. In Corollary~\ref{t:cyclicH1} we deal with a particular case in which the character automatically extends and in Corollary~\ref{c:ribbon} we give a bound on fusion number of ribbon knots. 

\begin{cor}\label{t:cyclicH1}
Let $Y$ be a $3$-manifold with $H_1(Y)$ cyclic of order $m^2$ and $\phi:H_1(Y) \to \Zk$ be a nontrivial character with $k | m$.
If $W$ is a rational homology ball with $\de W = Y$, then every handle decomposition of $W$ contains at least $|\sigma(Y,\phi)|-1$ odd-index handles.
\end{cor}
\begin{proof}
Using the long exact sequence for the pair $(W,Y)$, it is not difficult to show that the image of $i_*: H_1(Y)\to H_1(W)$ has order $m$. It follows that, whenever $k | m$, $\ker i_{*} \subset \ker\phi$, and hence $\phi$ factors through the image of $i_*$, giving $\psi_0: i_*(H_1(Y)) \to \Q/\Z$, where we look at $\Q/\Z$ as the set of roots of unity in $\C^*$.
Since $\Q/\Z$ is an injective $\Z$-module, we can extend $\psi_0$ to $H_1(W)$, hence obtaining an extension $\psi:H_1(W)\to \Q/\Z \subset \C^*$.

Therefore, the assumptions of Theorem~\ref{t:technical} are satisfied, and the result follows.
\end{proof}

We now turn to give a lower bound on the fusion number
of a ribbon knot. In what follows, given a knot $K\subset S^3$, we will denote with $\det K$ the determinant of $K$ and with $\Sigma(K)$ the double cover of $S^3$ branched over $K$.

\begin{cor}\label{c:ribbon}
Let $K$ be a ribbon knot with fusion number $b$. Then
\[
b \ge \min_H \max_\phi |\sigma(\Sigma(K),\phi)| - 1,
\]
where the minimum is taken over all subgroups $H < H_1(\Sigma(K))$ of index $\sqrt{|{\det K}|}$ and the maximum is taken over all characters $\phi$ whose kernel contains $H$.
\end{cor}
%
\begin{proof}
Let $D$ be a disk in the 4-ball realizing the fusion number for $K$.
The double cover of $B^{4}$ branched over $D$
is a rational homology ball $W$ bounded by $\Sigma(K)$, the double cover of $S^{3}$ branched over $K$.
By Theorem~\ref{t:technical}, from every character $\phi$ defined on $H_{1}(\Sigma(K))$ which extends over the rational homology ball we obtain a lower bound, $|\sigma(Y,\phi)|-1$, on the number of odd-index handles in any decomposition of $W$. 
Now, since $D$ is a ribbon disk, $W$ can be built with no 3-handles and with only $b$ 1-handles and thus, we obtain
\begin{equation}\label{e:ineq1}
b\ge \max\{|\sigma(Y,\phi)| - 1\,:\, \phi\ {\mathrm{factors\ through}}\ H_{1}(W)\}.
\end{equation}

It is well known that the order of $H_{1}(\Sigma(K))$ is equal to $|{\det K}|$ and, since there are no $3$-handles in $W$, the order of $H_{1}(W)$ equals $\sqrt{|{\det K}|}$ and 
the map $i_{*}: H_{1}(\Sigma(K))\rightarrow H_{1}(W)$ induced by the inclusion $i: \Sigma(K)\hookrightarrow W$ is a surjection.

Notice that the character $\phi$ factors through $H_{1}(W)$ if and only if $\ker i_{*}\subset\ker\phi$. For this to happen it is necessary that $\phi$ vanishes on a subgroup of order $\sqrt{|{\det K}|}$.
In order to give an obstruction on all possible rational homology balls, we need to minimize the right hand side in~\eqref{e:ineq1}  over all subgroups $H < H_1(\Sigma(K))$ of index $\sqrt{|{\det K}|}$.
\end{proof}

\begin{rmk}
A more general statement can be given in terms of cyclic branched covers of order a prime power $q = p^h$; that we denote with $\Sigma_q(K)$. In this more general setting, a ribbon disc with $b$ bands yields a rational homology 4-ball built with $(q-1)b$ 1-handles, and therefore one obtains
\begin{equation*}
b\ge \max_q \left\{\frac1{q-1} \min_H \max_\phi |\sigma(\Sigma_q(K),\phi)| - 1\right\},
\end{equation*}
where the outer maximum is taken over all prime powers $q$, the minimum is taken over all subgroups $H < H_1(\Sigma_q(K))$ of order divisible by $\sqrt{|H_1(\Sigma_q(K))|}$, and the inner maximum is taken over all characters $\phi$ that vanish on $H$.
\end{rmk}

If the character $\phi$ satisfying the assumptions of the above theorem is of prime power order, then by \cite{CassonGordon} we know that $|\sigma(\Sigma(K),\phi)|\leq 1$ and therefore we obtain no bound on $b$. Corollary~\ref{c:ribbon} is of interest when the order of the character is not a prime power and examples of non trivial bounds will be discussed in the next section.

We now focus on the special case of 3-manifolds $Y$ obtained as surgery on a knot $K\subset S^{3}$. We shall denote such manifolds as $S^{3}_{r}(K)$, where $r\in\Q$ is the surgery coefficient.
For this class of manifolds we will give a bound on the complexity of a rational homology ball bounded by $Y$ in terms of $\sigma_{K}(\omega_{m})$, the Levine--Tristram signature of the knot $K$ evaluated at a primitive root of unity of order $m$. The transition from Casson--Gordon invariants to the Levine--Tristram signature is done through work of Cimasoni--Florens~\cite[Theorem~6.7]{CimasoniFlorens}, which we now briefly recall. Consider a 3-manifold $Y$ obtained by surgery on a $\nu$-component framed link $L$ with linking matrix $\Lambda$ and a character $\phi:H_{1}(Y)\rightarrow \C^{*}$ mapping the meridian of the $i$-th component of $L$ to $\omega_{m}^{n_{i}}$, where $n_{i}$ is coprime to $m$. Set $\alpha=(\omega_{m}^{n_{1}},\dots,\omega_{m}^{n_{\nu}})$ and denote by $\sigma_{L}(\alpha)$ the coloured signature of $L$ evaluated on $\alpha$. Then, we have
\begin{equation}\label{e:CF}
\sigma(Y,\phi)=\sigma_{L}(\alpha)-\sum_{i<j}\Lambda_{ij}-\mathrm{sign}(\Lambda)+\frac{2}{m^{2}}\sum_{i,j}(m-n_{i})n_{j}\Lambda_{ij}.
\end{equation}

\begin{prop}\label{t:surgery}
If $S^3_{m^2}(K)$, with $m>1$, bounds a rational homology ball $W$ with one $1$-handle, then $|1-\sigma_K(e^{2a\pi i/m})-2a(m-a)| \le 1$ for every $1\le a < m$ such that $(a,m)=1$.
\end{prop}
\begin{proof}
Let us call $Y=S^3_{m^2}(K)$, and fix a character $\phi: H_1(Y)\to\C^*$ of order $m$.
Since $H_1(Y)$ is cyclic, it follows from the long exact sequence of the pair $(W,Y)$ that $\ker\phi = \ker(i_*:H_1(Y)\to H_1(W))$.
Therefore, $\phi$ extends to a character $\psi: H_1(W)\to \C^*$ (see the proof of Corollary~\ref{t:cyclicH1}).

Since $W$ is built with a single 1-handle, $\pi_1(W)=C_n$ is cyclic, and we can choose $\psi$ to be injective as follows.
Call $d = n/m$ the index of $i_*(H_1(Y))$ in $H_1(W)$, and consider the map $j: C_{m^2}\to C_n$, $j: 1 \mapsto d$.
We can fix identifications $H_1(Y) = C_{m^2}$ and $H_1(W) = C_n$ such that $i_*: H_1(Y) \to H_1(W)$ is identified with $j$, and $\phi$ maps $1\in C_{m^2}$ to $e^{2\pi i/m}$;
we can now choose the extension $\psi$ that maps $1$ to $e^{2\pi i/n}$, which is hence an isomorphism onto the group of $n$-th roots of unity in $\C^*$.

We will use $(W,\psi)$ to compute the Casson--Gordon signature invariant of the pair $(Y,\phi)$.
In this case in~\eqref{e:CGdef} we have $r=1$ and $\sigma(W)=0$, so $\sigma(Y,\phi)=\sigma^{\psi}(W)$.

We proceed now to estimate $|\sigma^{\psi}(W)|$.
Since $\psi$ is injective, and $\pi_1(W)$ is abelian, the cover associated to $\psi$ is the universal cover of $W$, and hence $H_1^\psi(W; \Q(\zeta_n)) = 0$.
The long exact sequence for the pair $(W,Y)$, twisted with $\psi$, gives:
\[
0 = H_1^\psi(W; \Q(\zeta_n)) \longrightarrow H_1^\psi(W,Y; \Q(\zeta_n)) \longrightarrow H_0^\psi(Y; \Q(\zeta_n)) = H_0^\phi(Y; \Q(\zeta_m))^{\oplus d},
\]
and the latter group vanishes since $\phi$ is a non-trivial character of $Y$.
Therefore, $H_1^\psi(W,Y;\Q(\zeta_n))$ is trivial, which in turn implies, by Poincar\'e--Lefschetz duality, that $H^3_\psi(W;\Q(\zeta_n)) = 0$, and hence $b^\psi_3(W; \Q(\zeta_n)) = 0$.
Finally, since $\psi$ is non trivial, we have $b^\psi_0(W)=0$. Now, since $W$ is a rational homology ball,
\[
1=\chi(W)=\chi^\psi(W) = b^\psi_0(W)-b^\psi_1(W)+b^\psi_2(W)-b^\psi_3(W)=b^\psi_2(W)
\]
and we obtain that $\dim H^\psi_2(W;\Q(\zeta_{m})) = 1$ and hence the signature $\sigma^\psi(W)$ of the equivariant intersection form is bounded by 1 in absolute value and thus $|\sigma(Y,\phi)|\le 1$. 

To finish the proof, we rewrite the Casson--Gordon signature invariant of $Y$ in terms of the Levine--Tristram signature of the surgery knot $K$ using~\eqref{e:CF}. To this end, identify $\Zm$ with the cyclic group generated by $\omega_{m}=e^{2i\pi/m}\in\C^{*}$ by sending $1\in\Zm$ to $e^{2i\pi/m}$ and denote by $\sigma_{K}(\cdot)$ the Levine--Tristram signature of the knot $K$. Recall that for knots the coloured and the Levine--Tristram signatures coincide. Let $\omega_{m}^{a}$ be the image of the meridian of $K$ under the character $\phi$. Finally, notice that the linking matrix of a framed knot is simply given by the framing. The statement of the proposition then follows from equation~\eqref{e:CF}, which in this simple case reads:
\[
|\sigma(Y,\phi)|=\left|\sigma_{K}(\omega_{m}^{a})-1+{\textstyle \frac{2}{m^{2}}}(m-a)am^{2}\right|.\qedhere
\]
\end{proof}

\begin{rmk}
In fact, the key property used in the proof is that $\pi_1(W)$ is cyclic; hence the statement holds under this assumption as well.
\end{rmk}

When a rational homology sphere $Y$ bounds a rational homology ball $W$, one can give a lower bound on the complexity of $W$ by looking at $H_1(Y)$.

\begin{prop}
If $Y$ bounds a rational homology ball $W$, and $H_1(Y)$ is generated by no fewer than $g$ generators, every handle decomposition of $W$ contains at least $\lceil g/2 \rceil$ $1$-handles.
\end{prop}

\begin{proof}[Proof (sketch)]
Take a handle decomposition of $W$ with a single $0$-handle,
$n_1$ $1$-handles, $n_2$ $2$-handles, and $n_3$ $3$-handes.
Since $W$ is a rational homology ball, $n_2 = n_1+n_3$.
Consider the $4$-handlebody $W'$ obtained by attaching only the $1$- and $2$-handles of $W$:
by construction, $Y' := \de W' = Y\# n_3(S^1\times S^2)$, and therefore $H_1(Y') = H_1(Y) \oplus \Z^{n_3}$.
Now perform a dot-zero surgery along the core of each $1$-handle.
This presents $Y'$ as an integer surgery along a $(n_1+n_2)$-component link, and correspondingly presents $H_1(Y')$ as a quotient of $\Z^{n_1+n_2}$;
it follows that $n_1+n_2 \ge g + n_3$, hence $g \le n_1+n_2-n_3 = 2n_1$.
\end{proof}

The statement of Proposition~\ref{t:surgery} can be extended to surgeries with rational coefficients. As shown in Figure~\ref{f:rational-to-integral}, a rational surgery on a knot $K$ can be interpreted as an integral surgery on a link $L=K\cup U_{2}\cup\cdots\cup U_{n}$ where all the $U_{i}$'s are unknots. We will use this link $L$ to compute the Casson--Gordon signature invariants of $Y=S^{3}_{m^{2}/{q}}(K)$. Notice that any character $\phi:H_{1}(Y)\rightarrow\Zm$ can be determined from a character defined on  $H_{1}(S^{3}\setminus L)$ sending the meridian of $K$ to $a\in\Zm$ and extending to $H_{1}(Y)$. If in the link $L$ we replace $K$ with an unknot $U_{1}$ and leave the same surgery coefficients, we obtain a surgery description of the lens space $L(m^{2},-q)=S^3_{m^2/q}(U)$ and a character $\chi_{a}:H_{1}(L(m^{2},-q))\rightarrow\Zm$ sending the meridian of $U_{1}$ to $a$. With all these conventions in place, we have the following statement. 

\begin{prop}\label{p:ratcoef}
If $Y=S^3_{m^2/q}(K)$ bounds a rational homology ball $W$ with one $1$-handle, then $|\sigma_K(e^{2i\pi a/m})+\sigma(L(m^{2},-q),\chi_{a})| \le 1$ for every $1\le a < m$ such that $(a,m)=1$.
\end{prop}

\begin{rmk}
There is an explicit formula for $\sigma(L(m^{2},-q),\chi_{a})$ given by Gilmer \cite[Example 3.9]{gilmer}.
\end{rmk}

\begin{proof}[Proof of Proposition~\ref{p:ratcoef}]
The same arguments used in Proposition~\ref{t:surgery} allow us to conclude in this case that any surjective character $\phi:H_{1}(Y)\rightarrow\Zm$ has an injective extension $\psi$ to $W$, $\sigma(Y,\phi)=\sigma^{\psi}(W)$ and, since $b_{2}^{\psi}(W)=1$, we have $|\sigma(Y,\phi)|\leq 1$.

To finish the proof, we want to express $\sigma(Y,\phi)$ using formula~\eqref{e:CF} applied to the surgery diagram depicted in Figure~\ref{f:rational-to-integral}. We refer the reader to \cite{CimasoniFlorens} for the pertinent definitions. The formula given by Cimasoni and Florens has one term that depends on the knot $K$, the colored signature of $L$, and all the others, which we will denote by $T_{\phi,\Lambda}$, depend exclusively on the image of the meridians of $L$ via $\phi$ and on the linking matrix $\Lambda$ of the surgery presentation of $Y$.
It follows that, with the exception of the first term in the formula, all the others remain unchanged if we substitute $K$ with an unknot.
That is, if we compute the Casson--Gordon invariant of a lens space from the chain surgery presentation with coefficients $(a_{1},\dots,a_{n})$ and for the character that is defined by sending the meridian of $U_{1}$ to $e^{2i\pi a/m}$. This Casson--Gordon invariant is precisely $\sigma(L(m^{2},-q),\chi_{a}$).

Now, notice that the link $L$ bounds an evident $C$-complex in the sense of~\cite{CimasoniFlorens} given by a Seifert surface for $K$ and a series of embedded disks, one for each unknot. The first homology of this complex coincides with the first homology of the Seifert surface for $K$, and the multivariable coloured signature of $L$  evaluated at any vector of roots of unity $(\omega_{1},\dots,\omega_{n})$ coincides with the Levine--Tristram signature of $K$ evaluated at $\omega_{1}$. 
This yields
\[
\sigma(Y,\phi)=\sigma_{L}(\omega_{1},\dots,\omega_{n})+T_{\phi,\Lambda}=\sigma_{K}(\omega_{1})+T_{\phi,\Lambda}.
\]
Since there is an evident contractible $C$-complex for the chain surgery presentation of $L(m^{2},-q)$, it follows that $\sigma(L(m^{2},-q),\chi_{a})=T_{\phi,\Lambda}$ and, by definition of $\phi$, we have $\omega_{1}=e^{2i\pi a/m}$. The result follows.
\end{proof}

\begin{figure}
\labellist
\pinlabel $K$ at 28 49
\pinlabel $\dots$ at 236 51
\pinlabel $a_1$ at 103 100
\pinlabel $a_2$ at 130 93
\pinlabel $a_n$ at 340 93
\endlabellist
\centering
\includegraphics[scale=0.9]{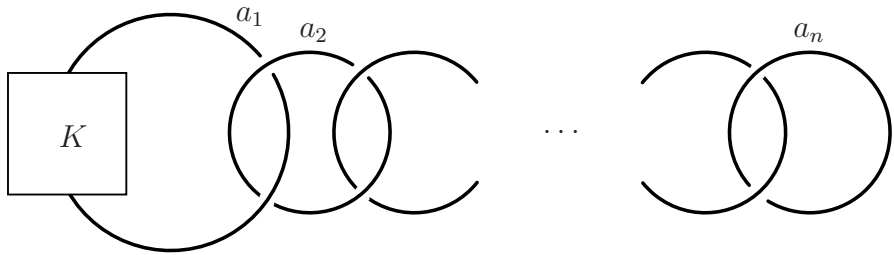}
\caption{The integral surgery picture for $p/q$-surgery along $K$, where $[a_1,\dots,a_n]^-$ is the negative continued fraction expansion of $p/q$.}
\label{f:rational-to-integral}
\end{figure}


\section{Examples}\label{examples}

\begin{example}
As promised in the introduction, we prove that the 3-manifold $Y = L(25,21)\# L(4,3)$ bounds no rational homology ball constructed with a single 1-handle.
However, it bounds a rational homology ball built with two 1-handles and two 2-handles.

Indeed, as shown by Moser~\cite{Moser}, $S^3_{100}(T_{4,25}) = L(25,21)\#L(4,3)$; we can now use Proposition~\ref{t:surgery} to obstruct the existence of such a ball.
In fact, using the formula from~\cite{Litherland}, we see that $\sigma_{T_{4,25}}(e^{i\pi/5}) = -15$ (note that $e^{i\pi/5}$ is a root of the Alexander polynomial of $T_{4,25}$, thus explaining why the signature is odd), and therefore
\[
|\sigma(Y,\phi)| = |1-\sigma_{T_{4,25}}(e^{i\pi/5})-2(10-1)| = |1+15-18| = 2.
\]
Since each of $L(25,21)$ and $L(4,3)$ bounds a rational homology ball built with a single 1-handle and a single 2-handle~\cite{LecuonaBaker}, their connected sum does indeed bound a rational homology ball, built with two 1-handles and two 2-handles, namely the boundary connected sum of the two balls above.

Also, note that $L(25,7)$ is a lens space that bounds a rational homology ball $W$, and that $L(25,7)\#L(25,7)$ bounds a rational homology ball built with a single 1-handle and a single 2-handle, which is therefore simpler than the boundary connected sum of two copies of $W$.
This shows that the example above is nontrivial.
\end{example}

In the following examples we will be using the Fibonacci numbers, defined by 
\[
\left\{\begin{array}{l}
F_0 = 0,\\
F_1 = 1,\\
F_{n+1} = F_n+F_{n-1}.
\end{array}\right.
\]
\begin{example}\label{e:torus}
In fact, the previous example readily generalises to the following family: whenever $Y_{a,b} = L(a^2,-b^2)\#L(b^2,-a^2) = S^3_{(ab)^2}(T_{a^2,b^2})$ and $\lfloor b/a\rfloor \le 2$ bounds a rational homology ball, we will show that for the character $\phi$ on $Y_{a,b}$ that maps the meridian to $\exp(2\pi i/ab)$, $|\sigma(Y_{a,b},\phi)|\ge 2$, thus showing that $Y_{a,b}$ does not bound a rational homology ball with one 1-handle.

Note that if $ab$ is odd, $Y_{a,b}$ is the branched double cover of a ribbon knot by work of Lisca~\cite{Lisca-ribbon}, and for every character $\phi'$ of order a \emph{prime power}, $|\sigma(Y_{a,b},\phi')|\le 1$~\cite[Theorem 2]{CassonGordon}.

We now prove the claim above that $|\sigma(Y_{a,b},\phi)|\ge 2$. From equation~\eqref{e:CF} we know that
\[
|\sigma(Y_{a,b},\phi)| =|\sigma_{T_{a^2,b^2}}(e^{2\pi i/ab})-1+\frac{2}{a^{2}b^{2}}(ab-1)a^{2}b^{2}| =|\sigma_{T_{a^2,b^2}}(e^{2\pi i/ab}) + 2ab -3|,
\]
and therefore it suffices to show that $\sigma_{T_{a^2,b^2}}(e^{2\pi i/ab}) \ge -2ab+5$.
Observe that the Alexander polynomial of $T_{a^2,b^2}$ has simple roots, and they are the
$a^2b^2$-th roots of unity that are neither $a^2$-th nor $b^2$-th roots of 1.
By these two observations, we know that $\sigma_{T_{a^2,b^2}}$ can only jump by 2 at each such root of unity, and that the value at these roots of unity is the average of the neighbouring values (and in particular it is odd); since the signature vanishes at $1$, at $e^{2\pi i/{ab}}$ it is bounded by the number of roots in the arc $e^{2\pi it}$ with $t$ in the open interval $(0,1/ab)$.
These are easily counted to be $ab-1-\lfloor b/a\rfloor \ge ab-3$.
Since $e^{2\pi i/ab}$ is a root of the Alexander polynomial
\[
|\sigma_{T_{a^2,b^2}}(e^{2\pi i/ab})| \le 2(ab-3)+1.
\]
Note that, a posteriori, since $Y_{a,b}$ bounds a rational homology ball built with two 1-handles,
\[
|\sigma(Y_{a,b},\phi)| \le 3
\]
which also implies that $\sigma_{T_{a^2,b^2}}(e^{2\pi i/ab}) \le 6-2ab$, thus proving that $\sigma_{T_{a^2,b^2}}(e^{2\pi i/ab}) = 5-2ab$ and that $\sigma_{T_{a^2,b^2}}(e^{2\pi it})$ is non-increasing for $t$ in the closed interval $[0,1/ab]$.

As a concrete example, we can choose $a= F_5 = 5$, $b= F_7 = 13$;
in this case, $Y_{a,b}$ bounds a rational homology ball, obtained as the complement of a rational cuspidal curve in $\mathbb{CP}^2$ (see~\cite{Kashiwara, 55letters}).
In fact, there is a rational curve $C$ in $\mathbb{CP}^2$ whose unique singularity has link $T(25,169)$~\cite[Theorem 1.1(c)]{55letters};
the boundary of an open regular neighbourhood $N$ of $C$ is $Y_{5,13}$, and the complement of $N$ in $\mathbb{CP}^2$ is a rational homology ball (see~\cite{BorodzikLivingston}).
Additionally, $K_{5,13} = K(25,-169)\#K(169,-25)=K(25,6)\#K(169,144)$, and each of the two summands is ribbon with fusion number 1.
However, $K_{5,13}$ has fusion number 2: indeed, by Corollary~\ref{c:ribbon} its fusion number is at least 2, and since it is the connected sum of two fusion number-1 knots, the inequality is sharp. These examples show, once again, that the assumption that the order of $\phi'$ is a prime power in~\cite[Theorem 2]{CassonGordon} is indeed essential.

Indeed, infinitely many pairs of odd integers arise in this fashion: for each pair $(a,b) = (F_{6k-1}, F_{6k+1})$ the 3-manifold $Y_{a,b}$ bounds a rational homology ball~\cite{Kashiwara, 55letters} and $H_1(Y_{a,b})$ has odd order, since both $F_{6k-1}$ and $F_{6k+1}$ are odd.
From now on we restrict to $(a,b)$ belonging to this family, which has $\lfloor b/a\rfloor = 2$.
If $(a,b) = (F_{6k+1}, F_{6k+3})$ or $(a,b) = (F_{6k+3}, F_{6k+5})$, then $Y_{a,b}$ still bounds a rational homology ball, but in this case $|H_1(Y_{a,b})| = a^2b^2$ is even, since $F_{6k+3}$ is divisible by $F_3 = 2$.

Note that, when $ab$ is odd, $Y_{a,b}$ is the branched double cover of a knot $K_{a,b}$ which is a connected sum of two ribbon 2-bridge knots, and therefore, since the fusion number of a ribbon 2-bridge knot is 1~\cite{LecuonaBaker}, the fusion number of $K_{a,b}$ is 2.
In order to see this, we use the (elementary) identity $a^2-3ab+b^2 = -1$; from this it follows that $-a^2 \equiv 1 \pmod b$ and that the quotient $(-a^2-1)/b = -3a+b$ is coprime with $b$; symmetrically, $(-b^2-1)/a$ is an integer coprime with $a$.
That is, both 2-bridge knots $K(a^2,-b^2)$ and $K(b^2,-a^2)$ are of the form $K(m^2,km+1)$ for some $k$ coprime with $m$, and hence they are ribbon~\cite{Lisca-ribbon}.
Moreover, as observed by Baker, Buck, and the third author~\cite{LecuonaBaker}, they both bound a ribbon disc built with a single 1-handle, hence their fusion number is 1.
By taking the branched double cover of the boundary connected sum of these ribbon discs, one exhibits a rational homology ball bounding $Y_{a,b}$ built with two 1-handles and two 2-handles, and this is minimal, by Proposition~\ref{t:surgery}.
\end{example}

\begin{example}
One can refine the example above to produce an \emph{irreducible} surgery that bounds no rational homology ball with one 1-handle, but does bound one with three.
In fact, $Y = S^3_{400}(T_{4,25;2,201})$ bounds a rational homology ball~\cite{PaoloMarcoKyleAna}, and a quick computation with the Levine--Tristram signature using~\cite{Litherland} and Proposition~\ref{t:surgery} yields
\[
|\sigma(Y,\phi)| = |1-\sigma_{T_{4,25;2,201}}(e^{i\pi/10})-2\cdot (20-1)| = |\sigma_{T_{2,201}}(e^{i\pi/10}) - 37| = |35-37| = 2,
\]
where $\phi$ is the character that sends the meridian to $e^{2i\pi/20}$.

Irreducibility is proven by looking at the canonical plumbing diagram for $Y$;
since it is connected, $Y$ is irreducible (see~\cite{Neumann,EisenbudNeumann} for details).

We can also find examples when $H_1(Y)$ has odd order. Indeed, one can look at $Y = S^3_{9\cdot25\cdot169}(T_{25,169;3,3\cdot25\cdot169+1})$ and the character $\phi$ that sends a meridian to $e^{2\pi i/(3\cdot 5\cdot 13)}$; a similar computation to the one above with the Levine--Tristram signatures yields:
\[
|\sigma(Y,\phi)| = \left|1-\sigma_{T_{25,169;3,12676}}(e^{2\pi i/195}) - 388\right| = 2,
\]
hence proving that $Y$ does not bound a rational homology ball with a single 1-handle.
\end{example}

We conclude with a rather lengthy example where we produce a family of irreducible 3-manifolds with cyclic first homology group. Each of these manifolds bounds a rational homology ball built with handles of index at most 2, but such that the number of handles needed is arbitrarily large.

The structure of the argument is the following: we fix an integer $v$, and we build a 3-manifold $Y$ by a construction that is akin to the plumbing of spheres. The manifold $Y$ will depend on the choice of $v$ knots $K_1,\dots,K_v\subset S^3$ and $v+1$ integers $a,n_1,\dots,n_v$, and we show that, under certain assumptions, all these manifolds have cyclic $H_1$.
We then specialise to a certain family of knots $K_j$ and integers $n_j$, coming from the example above, and we prove that the resulting $Y$ does indeed bound a rational homology ball built out of $2v+1$ 1-handles and $2v+1$ 2-handles.
We then compute the signature defect associated to a certain character $\phi$ of $Y$ within an error of 2, and using Corollary~\ref{t:cyclicH1}
we show that any rational homology ball 2-handlebody needs at least $2v-1$ 1-handles.
Finally, we argue the irreducibility of $Y$.

\begin{example}
Let us consider the following (modified) plumbing diagram, representing a 3-manifold $Y$:
\[
 \xygraph{
!{<0cm,0cm>;<1cm,0cm>:<0cm,1cm>::}
!~-{@{-}@[|(2.5)]}
!{(-1,0) }*+{\bullet}="c"
!{(1,0) }*+{\bullet}="d"
!{(3,0) }*+{\dots}="e"
!{(5,0) }*+{\bullet}="x"
!{(7,0) }*+{\bullet}="g"
!{(3,1.5) }*+{\bullet}="h"
!{(3,3) }*+{\bullet}="ac"
!{(1.5,3) }*+{\bullet}="al"
!{(4.5,3) }*+{\bullet}="ar"
!{(-1,-0.4) }*+{[K_1, n_1^2]}
!{(1,-0.4) }*+{[K_2, n_2^2]}
!{(5,-0.4) }*+{[K_{v-1}, n_{v-1}^2]}
!{(7,-0.4) }*+{[K_{v}, n_{v}^2]}
!{(3.4,1.7) }*+{a}
!{(4.5,3.4) }*+{-2}
!{(1.5,3.4) }*+{-2}
!{(3,3.4) }*+{-1}
"c"-"h"
"d"-"h"
"h"-"x"
"h"-"g"
"h"-"ac"
"ac"-"al"
"ac"-"ar"
}
\]
Where each label $[K,n]$ at the bottom signifies that, in the corresponding surgery picture for $Y$, instead of an unknot we use the knot $K$
the knot $K$, with framing $n$.
In other words, instead of plumbing sphere bundles, we plumb the trace of $n$-surgery along $K$ using the co-core of the 2-handle.

Whenever $S^3_{n_j^2}(K_j)$ bounds a rational homology ball for each $j = 1,\dots, v$, so does $Y$ (\cite{Paolo}; see also~\cite{PaoloMarco}).
We claim that if each $n_j$ is odd, $n_j$ and $n_k$ are pairwise coprime for each $j\neq k$, and $a\not\equiv v \pmod 2$, then $H_1(Y)$ is cyclic. From now on, we will make these assumptions on $n_j$ and $a$ and we shall specialise both $K_j$ and $n_j$ later.

After two subsequents blowdowns we obtain the following diagram:
\[
 \xygraph{
!{<0cm,0cm>;<1cm,0cm>:<0cm,1cm>::}
!~-{@{-}@[|(2.5)]}
!{(-1,0) }*+{\bullet}="c"
!{(1,0) }*+{\bullet}="d"
!{(3,0) }*+{\dots}="e"
!{(5,0) }*+{\bullet}="x"
!{(7,0) }*+{\bullet}="g"
!{(3,1.5) }*+{\bullet}="h"
!{(3,3) }*+{\bullet}="ac"
!{(-1,-0.4) }*+{[K_1, n_1^2]}
!{(1,-0.4) }*+{[K_2, n_2^2]}
!{(5,-0.4) }*+{[K_{v-1}, n_{v-1}^2]}
!{(7,-0.4) }*+{[K_{v}, n_{v}^2]}
!{(3.8,1.7) }*+{a+2}
!{(3,3.4) }*+{0}
"c"-"h"
"d"-"h"
"h"-"x"
"h"-"g"
"h"-@/^/ "ac"
"h"-@/_/ "ac"
}
\]
Note that the double edge between the two top-most vertices is not to be intended in the plumbing diagram sense, but rather signifies a double linking between the corresponding attaching circles; in particular, it \emph{does not} increase $b_1$. Let $P$ be the 4-manifold associated to the diagram above, with $\partial P = Y$.

Notice that $P$ is a 2-handlebody, i.e. it is obtained from $B^4$ by attaching only 2-handles, and hence $H_1(P;R) = 0$ and $H_2(P;R) = H_2(P)\otimes R$, and $H_2(P,Y;R) = H_2(P,Y)\otimes R$ for each ring $R$, and that the latter are both free over $R$ of rank $v+2$.

We now set out to compute $H_1(Y)$ as the quotient of $\Z^{v+2}$ by the image of the intersection matrix of the link coming from the diagram above. The matrix is
\[
Q = \left(\begin{array}{ccccc}
a+2 & 2&1&\cdots&1\\
2 & 0&0&\cdots&0\\
1 & 0&n_1^2&&\\
\vdots & \vdots&&\ddots&\\
1 & 0&&&n_v^2
\end{array}\right).
\]

By expanding along the second row, we easily see that $|{\det Q}| = (\prod_{j=0}^v n_j)^2$, where we let $n_0 = 2$ for convenience.
 
In order to see that $H_1(Y)$ is cyclic, since $|{\det Q}| = \prod_{j=0}^v |\Z/n_j^2\Z|$, it is enough to check that, for each $j$, $H_1(Y;\Z/n_j^2\Z)\cong \Z/n_j^2\Z$.
Let $R = \Z/n_j^2\Z$.

The long exact sequence for the pair $(P,Y)$ yields:
\[
\xymatrix{
H_2(P;R) \ar[d]^\cong \ar[r] & H_2(P,Y;R) \ar[r]\ar[d]^\cong&  H_1(Y;R)\ar[r]\ar[d]^\cong & 0\\
R^{v+2} \ar[r]^{Q_R} & R^{v+2}\ar[r]&  \coker Q_R\ar[r] & 0
}
\]
where $Q_R$ is the reduction of $Q$ modulo $n_j^2$. It is enough to show that the quotient is cyclic.
This is elementary from the matrix $Q$, since $n_k$ is now invertible in $R$ for every $k\neq j$.
When $j=0$, it is helpful (but not necessary) to use the fact that $n_k^2 \equiv 1 \pmod 4$ for each $k>0$; one then reduces to the case of the matrix $\left(\begin{array}{cc}
a+2-v & 2\\
2 & 0
\end{array}\right)$, which is well-known to have cyclic cokernel precisely when $a+2-v$ is odd.

Observe that, after doing a dot-zero surgery on the 0-framed unknot, the diagram above also exhibits a rational homology cobordism $W$ from $Y' := \#_{j=1}^{v}S^3_{n_j^2}(K_j)$ to $Y$.
Moreover, it is easy to check that the inclusion induces an injection $i'_*: H_1(Y')\to H_1(W)$, hence every character $\phi'$ of $Y'$ extends to $W$, and we can further restrict it to $Y$.
Let $i_*: H_1(Y)\to H_1(W)$ be the map induced by the inclusion.
Additionally, since $\coker i'_* = \coker i_* = \Z/2\Z$ and since $|H_1(Y')|$ is odd, the order of the induced character on $Y$ is either the order of $\phi'$ or twice as large.

We are going to look at a character $\phi: H_1(Y)\to \C^*$ induced as above from the character $\phi'$ on $Y'$ that sends the meridian of $K_j$ to $\exp(2\pi i/n_j^2)$ for each $j$.

By additivity of the Casson--Gordon signature defects~\cite{Gilmer-additive},
\[
|\sigma(Y',\phi')| = \left\lvert\sum_{j=1}^n \sigma\big(S^3_{n_j^2}(K_j),\phi_j\big)\right\rvert = \left\lvert\sum_{j=1}^n \left(1 - \sigma_{K_j}\big(e^{2\pi i/n_j}\big) + 2(n_j-1)\right)\right\rvert.
\]
As in Example~\ref{e:torus}, we turn our attention to torus knots and Fibonacci numbers and from now on we assume that $K_j = T_{F_{p_j}^2, F_{p_j+2}^2}$, $n_j = F_{p_j}F_{p_j+2}$, where the sequence $p_j$ is defined recursively by $p_1 = 5$, $p_{j+1} = 6\prod_{k\le j} (p_k^2+2p_k)-1$.

Since $\gcd(F_a,F_b) = F_{\gcd(a,b)}$, and since $F_3 = 2$, we have that $n_j$ is odd for each $j$.
Moreover, by construction,
\[
p_j \equiv 1 \pmod 2, \quad p_j \equiv -1 \pmod {p_k}, \quad p_j \equiv -1 \pmod {p_k+2},
\]
and hence both $p_j$ and $p_j+2$ are odd and coprime with $p_k$ for each $k$; thus, $n_j$ and $n_k$ are coprime, too.

It follows from work of Kashiwara~\cite{Kashiwara} (see also~\cite{55letters, BorodzikLivingston}) that $S^3_{n_j^2}(K_j)$, which we denoted by $Y_{F_{p_{j}},F_{p_{j}+2}}$ in Example~\ref{e:torus}, bounds a rational homology ball.
Moreover, Example~\ref{e:torus} shows that we can choose these balls to be constructed using only 1- and 2-handles.
Since $W$ is a rational homology cobordism constructed with one 1-handle and one 2-handle, $Y$ also bounds a rational homology ball constructed without 3-handles. In fact, the rational homology balls constructed in Example~\ref{e:torus} used two 1-handles, and $W$ uses only one 1-handle, so $Y$ bounds a rational homology ball constructed with $2v+1$ 1-handles and $2v+1$ 2-handles.

In Example~\ref{e:torus} we estimated the Casson--Gordon signature and obtained 
\[
\left|\sigma(Y_{F_{p_{j}},F_{p_{j}+2}},\phi_{j})\right|\ge 2,
\]
which combined with the additivity under connected sum yields  $|\sigma(Y',\phi')| \geq 2v$.

Since $\phi'$ extends to the cobordism $W$, we can glue $W$ to any 4-manifold $Z'$ to which $\phi'$ extends (rationally), and use the resulting 4-manifold $Z = W\cup Z'$ to compute the signature defect $\sigma(Y,\phi)$.

Since $W$ is a rational homology cobordism, the ordinary signature does not change; that is, $\sigma(Z) = \sigma(Z')$. The twisted signature $\sigma^\psi(Z)$ is also controlled by $\sigma^{\psi'}(Z')$: indeed, since $W$ contains a single 2-handle, $|\sigma^\psi(Z) -\sigma^{\psi'}(Z')| \le 1$ by Novikov additivity.

It follows that $|\sigma(Y,\phi)| \ge |\sigma(Y',\phi')-1|\ge 2v-1$, and therefore any rational homology ball filling $Y$, that is built only using 1- and 2-handles, has at least $2v-1$ 1-handles, by Corollary~\ref{t:cyclicH1}.

To conclude, we argue that $Y$ is irreducible. Indeed, we can replace each of the nodes labelled with $[K_j, n_j^2]$ above with a negative definite plumbing tree,
and, using Neumann's criterion~\cite{Neumann}, one can check that the plumbing is in normal form, and its boundary is irreducible.
%

\end{example}

\bibliographystyle{amsplain}
\bibliography{casson-gordon}
\end{document}